\documentclass[12pt]{amsart}
\usepackage{amscd,verbatim}

\newcommand{\A}{\mathbf{A}}
\newcommand{\F}{\mathbf{F}}
\newcommand{\G}{\mathbb{G}}
\renewcommand{\P}{\mathbf{P}}
\newcommand{\Z}{\mathbf{Z}}
\newcommand{\sF}{\mathcal{F}}
\newcommand{\sG}{\mathcal{G}}

\renewcommand{\phi}{\varphi}

\newcommand{\Hom}{\operatorname{Hom}}
\newcommand{\uHom}{\operatorname{\underline{Hom}}}
\newcommand{\Ext}{\operatorname{Ext}}
\newcommand{\uExt}{\operatorname{\underline{Ext}}}
\newcommand{\HI}{{\operatorname{HI}}}
\newcommand{\NST}{{\operatorname{NST}}}
\newcommand{\DM}{\operatorname{DM}}

\newcommand{\tr}{{\operatorname{tr}}}
\newcommand{\et}{{\operatorname{\acute{e}t}}}
\newcommand{\eff}{{\operatorname{eff}}}
\renewcommand{\o}{{\operatorname{o}}}
\newcommand{\Alb}{\operatorname{Alb}}
\newcommand{\Pic}{\operatorname{Pic}}
\newcommand{\NS}{\operatorname{NS}}
\newcommand{\Ker}{\operatorname{Ker}}
\newcommand{\Coker}{\operatorname{Coker}}
\newcommand{\Spec}{\operatorname{Spec}}
\newcommand{\rank}{\operatorname{rk}}

\newcommand{\by}[1]{\overset{#1}{\longrightarrow}}
\newcommand{\iso}{\by{\sim}}
\newcommand{\tto}{\dashrightarrow}
\newcommand{\surj}{\rightarrow\!\!\!\!\!\rightarrow}

\newtheorem{prop}{Proposition}
\newtheorem{thm}{Theorem}
\newtheorem{cor}{Corollary}
\newtheorem{lemma}{Lemma}
\theoremstyle{remark}
\newtheorem{rk}{Remark}
\newtheorem{qn}{Question}
\newtheorem{ex}{Example}

\newcounter{spec}
\newenvironment{thlist}{\begin{list}{\rm{(\roman{spec})}}%
{\usecounter{spec}\labelwidth=20pt\itemindent=0pt\labelsep=10pt}}%
{\end{list}}%

\begin{document}
\title{Algebraic tori as Nisnevich sheaves with transfers}
\author{Bruno Kahn}
\address{Institut de Math\'ematiques de Jussieu\\UMR 7586\\ Case 247\\4 place
Jussieu\\75252 Paris Cedex 05\\France}
\email{kahn@math.jussieu.fr}
\date{March 9, 2012}
\begin{abstract}
We relate $R$-equivalence on tori with Voevodsky's theory of homotopy invariant Nisnevich
sheaves with transfers and effective motivic complexes.
\end{abstract}
\maketitle

\tableofcontents

\section{Main results} Let $k$ be a field and let $T$ be a $k$-torus. The 
$R$-equivalence classes on $T$ have been extensively studied by several authors, notably by 
Colliot-Th\'el\`ene and Sansuc in a series of papers including \cite{cts} and \cite{cts2}: they
play a central r\^ole in many rationality issues. In this note, we show that Voevodsky's triangulated category of motives sheds a new light on
this question: see Corollaries \ref{c3}, \ref{c4} and \ref{c5} below. 

More generally, let
$G$ be a semi-abelian variety over
$k$, which is an extension of an abelian variety $A$ by a torus $T$. Denote by
$\HI$ the category of homotopy invariant Nisnevich sheaves with transfers over $k$ in the
sense of Voevodsky \cite{voetri}. Then $G$ has a natural structure of an object of $\HI$
(\cite[proof of Lemma 3.2]{spsz}, \cite[Lemma 1.3.2]{bvk}).    Let
$L$ be the group of cocharacters of $T$.  

\begin{prop}\label{p2} There is a natural isomorphism $G_{-1}\iso L$ in $\HI$.
\end{prop}

Here $_{-1}$ is the contraction operation of \cite[p. 96]{voepre}, whose definition is
recalled in the proof below.

\begin{proof} Recall that if $\sF$ is a presheaf [with transfers] on smooth
$k$-schemes, the presheaf [with transfers] $\sF_{-1}^p$ is defined by
\[U\mapsto \Coker(\sF(U\times \A^1)\to \sF(U\times \G_m)). \]

If $\sF$ is homotopy invariant, we may replace $U\times \A^1$ by $U$ and the rational point
$1\in \G_m$ realises $\sF_{-1}^p(U)$ as a functorial direct summand of $\sF(U\times \G_m)$.

If $\sF$ is a Nisnevich sheaf [with transfers], $\sF_{-1}$ is defined as the sheaf associated
to $\sF_{-1}^p$.

Now $A(U\times \A^1)\iso A(U\times \G_m)$ since $A$ is an abelian variety, hence $A^p_{-1}=0$.
We therefore have an isomorphism of presheaves $T^p_{-1}\iso G^p_{-1}$, and \emph{a fortiori}
an isomorphism of Nisnevich sheaves $T_{-1}\iso G_{-1}$.

Let $p:\G_m\to \Spec k$ be the structural map. One easily checks that the \emph{\'etale} sheaf
$\Coker(T\by{i} p_*p^*T)$ is canonically isomorphic to $L$. Since $i$ is split, its cokernel is
still $L$ if we view it as a morphism of presheaves, hence of Nisnevich sheaves.
\end{proof}

From now on, we assume $k$ perfect. Let $\DM^\eff_-$ be the triangulated category of effective
motivic complexes introduced in
\cite{voetri}: it has a $t$-structure with heart $\HI$. It also has a tensor structure and a
(partially defined) internal Hom. We then have an isomorphism
\[L[0]=G_{-1}[0]\simeq \uHom_{\DM_-^\eff}(\G_m[0],G[0])\]
\cite[Rk. 4.4]{sk}, hence by adjunction a morphism in $\DM_-^\eff$
\begin{equation}\label{eq2bis}
L[0]\otimes \G_m[0]\to G.
\end{equation}

Let  $\nu_{\le 0} G[0]$ denote the cone of \eqref{eq2bis}: by \cite[Lemma 6.3]{birat} or \cite[\S 2]{motiftate}, $\nu_{\le 0} G[0]$ is the
\emph{birational motivic complex} associated to $G$. We want to compute
its homology sheaves.

For this, consider a coflasque resolution
\begin{equation}\label{eq3}
0\to Q\to L_0\to L\to 0
\end{equation}
of $L$ in the sense of \cite[p. 179]{cts}. Taking a coflasque resolution of $Q$ and
iterating, we get a resolution of $L$ by invertible lattices\footnote{Recall that  a
\emph{lattice} is a free finitely generated Galois module;  a lattice is \emph{invertible} if
it is a direct summand of a permutation lattice.}:
\begin{equation}\label{eq5}
\dots\to L_n\to \dots\to L_0\to L\to 0.
\end{equation}

We set
\[Q_n=\begin{cases}
Q& \text{for $n=1$}\\
\Ker(L_{n-1}\to
L_{n-2})& \text{for $n>1$}.
\end{cases}\]

\begin{thm}\label{t3} a) Let $T_n$ denote the torus with cocharacter group $L_n$. Then 
$\nu_{\le 0} G[0]$ is isomorphic to the complex
\[
\dots\to T_n\to \dots\to T_0\to G\to 0.
\]
b) Let $S_n$ be the torus with cocharacter group $Q_n$.
For any connected
smooth $k$-scheme $X$ with function field $K$, we have
\[H_n(\nu_{\le 0}G[0])(X)=
\begin{cases}
0 &\text{if $n<0$}\\
G(K)/R &\text{if $n=0$}\\
S_n(K)/R &\text{if $n>0$.}\\
\end{cases}\]
\end{thm}

The proof is given in Section \ref{s2}.

\begin{cor}\label{c3} The assignment $Sm(k)\ni X\mapsto \bigoplus_{x\in X^{(0)}} G(k(x))/R$
provides $G/R$ with the structure of a homotopy invariant Nisnevich sheaf with transfers. In
particular, any morphism $\phi:Y\to X$ of smooth connected $k$-schemes induces a morphism
$\phi^*:G(k(X))/R\to G(k(Y))/R$.\qed
\end{cor}

This functoriality is essential to formulate Theorem \ref{t4} below. For $\phi$ a closed
immersion of codimension $1$, it recovers a specialisation map on $R$-equivalence classes with
respect to a discrete valuation of rank $1$ which was obtained (for tori) by completely
different methods,
\emph{e.g.}  \cite[Th. 3.1 and Cor. 4.2]{cts2} or \cite{gille2}. (I am indebted to
Colliot-Th\'el\`ene for pointing out these references.)

\begin{cor}\label{c2} a) If $k$ is finitely generated, the $n$-th homology sheaf of $\nu_{\le
0} G[0]$ takes values in finitely generated abelian groups, and even in finite groups if $n>0$
or $G$ is a torus.\\
 b) If $G$ is a torus, then  $\nu_{\le 0} G[0]=0$ if $G$ is split by a Galois extension $E/k$
whose Galois group has cyclic Sylow subgroups. This condition is automatic if $k$ is (quasi-)finite.
\end{cor}

The proof is also given in Section \ref{s2}.

Given two semi-abelian varieties
$G,G'$, we would now like to understand the maps
\[\Hom_k(G,G')\to \Hom_{\DM_-^\eff}(\nu_{\le 0} G[0],\nu_{\le 0}G'[0])\to \Hom_\HI(G/R,G'/R).\]

In Section \ref{s4}, we succeed in elucidating the nature of their composition to a large extent, at least if $G$ is
a torus. Our main result, in the spirit of Yoneda's lemma, is

\begin{thm} \label{t4} Let $G,G'$ be two semi-abelian varieties, with $G$ a torus. 
Suppose
given, for every function field $K/k$, a homomorphism $f_K:G(K)/R\to G'(K)/R$ such that $f_K$ is
natural with respect to the functoriality of Corollary \ref{c3}. Then\\
a)  There exists an extension $\tilde G$ of $G$ by a permutation torus, and a homomorphism
$f:\tilde G\to G'$ inducing $(f_K)$.\\
b) $f_K$ is surjective for all $K$ if and only if there exist extensions $\tilde G,\tilde G'$ of
$G$ and $G'$ by permutation tori such that $f_K$ is induced by a split surjective homomorphism
$\tilde G\to \tilde G'$. 
\end{thm}

The proof is given in \S  \ref{proof of t4}. See Proposition
\ref{p4}, Corollary \ref{c6}, Remark \ref{r1} and Proposition \ref{p6} for complements.

This relates to questions of stable birationality studied by
Colliot-Th\'el\`ene and Sansuc in
\cite{cts} and \cite{cts2}, providing alternate proofs and strengthening of some of their
results (at least over a perfect field). More precisely:

\begin{cor}\label{c4} a) Let $G'$ be a semi-abelian $k$-variety such that $G'(K)/R\allowbreak =0$ for any
function field
$K/k$. Then $G'$ is an invertible torus.\\
b) In Theorem \ref{t4} b), assume that $f_K$ is bijective for all $K/k$. Then there exist
extensions $\tilde G$, $\tilde G'$  of $G$ and $G'$ by invertible tori such that $f_K$ is
induced by an isomorphism $\tilde G\iso \tilde G'$.
\end{cor}

\begin{proof} a) This is the special case $G=0$ of Theorem \ref{t4} b).

b) By Theorem \ref{t4} b), we may replace $G$ and $G'$ by extensions by permutation tori such
that $f_K$ is induced by a split surjection $f:G\to G'$. Let $T=\Ker f$. Then $T/R=0$
universally. By a), $T$ is invertible. 
\end{proof}

Corollary \ref{c4} a) is a version of \cite[Prop. 7.4]{cts2} (taking \cite[p. 199,
Th. 2]{cts} into account). Theorem
\ref{t4} was inspired by the desire to understand this result from a different viewpoint. 

\begin{cor}\label{c5} Let $f:G\tto G'$ be a rational map of semi-abelian varieties, with $G$ a
torus. Then the following conditions are equivalent:
\begin{thlist}
\item $f_*:\nu_{\le 0}G[0]\to \nu_{\le 0} G'[0]$ is an isomorphism (see Proposition \ref{p4}).
\item $f_*:G(K)/R\to G'(K)/R$ is bijective for any function field $K/k$.
\item $f$  is an isomorphism, up to extensions of $G$ and $G'$ by invertible tori and up to a
translation. (See Lemma \ref{l11}.) \qed 
\end{thlist}
\end{cor}

\subsection*{Acknowledgements} Part of Theorem \ref{t3} was obtained in the
course of discussions with Takao Yamazaki during his stay at the IMJ in October 2010: I would
like to thank him for inspiring exchanges. I also thank Daniel Bertrand for a helpful
discussion. Finally, I wish to acknowledge inspiration from the work of Colliot-Th\'el\`ene and
Sansuc, which will be obvious throughout this paper.

\section{Proofs of Theorem \ref{t3} and Corollary \ref{c2}}\label{s2}

\begin{lemma}\label{l1} The exact sequence
\[0\to T(k)\to G(k)\to A(k)\]
induces an exact sequence
\[0\to T(k)/R\by{i} G(k)/R\to A(k).\]
\end{lemma}

\begin{proof} Let $f:\P^1\tto G$ be a $k$-rational map defined at $0$ and $1$. Its composition
with the projection $G\to A$ is constant: thus the image of $f$ lies in a $T$-coset of $G$
defined by a rational point. This implies the injectivity of $i$, and the rest is clear.
\end{proof}

Let $\NST$ denote the category of Nisnevich sheaves with transfers. Recall that
$\DM_-^\eff$ may be viewed as a localisation of $D^-(\NST)$, and that its tensor structure is a
descent of the tensor structure on the latter category \cite[Prop.
3.2.3]{voetri}. 

\begin{lemma}\label{l3} If $G$ is an invertible torus, there is a canonical isomorphism in
$D^-(\NST)$
 \[L[0]\otimes \G_m\iso G[0].\]
In particular, $\nu_{\le 0} G[0]=0$.
\end{lemma}

\begin{proof} We reduce to the case $T=R_{E/k}\G_m$, where $E$ is a finite 
extension of $k$. Let us write more precisely $\NST(k)$ and $\NST(E)$. There is a pair of adjoint
functors
\[\NST(k)\by{f^*} \NST(E),\quad \NST(E)\by{f_*} \HI(k)\]
where $f:\Spec E\to \Spec k$ is the projection. Clearly,
\[f_*\Z = \Z_\tr(\Spec E),\quad f_*\G_m = T\]
where $\Z_\tr(\Spec E)$ is the Nisnevich sheaf with transfers represented
by $\Spec E$. Since $\Z_\tr(\Spec E) = L$, this proves the claim.
\end{proof}

\begin{proof}[Proof of Theorem \ref{t3}] a) Recall that $L_0$ is an invertible lattice chosen so that $L_0(E)\to L(E)$ is surjective for any extension $E/k$. In particular, \eqref{eq3} and \eqref{eq5} are
exact as sequences of  Nisnevich sheaves; hence
$L[0]$ is isomorphic in $D^-(\NST)$ to the complex
\[L_\cdot = \dots\to L_n\to \dots\to L_0\to 0.\]

(We may view \eqref{eq5} as a version of Voevodsky's ``canonical resolutions" as in \cite[\S 3.2
p. 206]{voetri}.)

By Lemma \ref{l3}, $L_n[0]\otimes \G_m[0]\simeq T_n[0]$ is homologically
concentrated in degree $0$ for all $n$. It follows that the complex
\[T_\cdot = \dots\to T_n\to \dots\to T_0\to 0\]
is isomorphic to $L[0]\otimes \G_m[0]$ in $D^-(\NST)$, hence \emph{a fortiori} in $\DM_-^\eff$.

b) For any nonempty open subscheme $U\subseteq
X$ we have isomorphisms
\begin{equation}\label{eq6}
H_n(\nu_{\le 0}G[0])(X)\iso H_n(\nu_{\le 0}G[0])(U)\iso H_n(\nu_{\le
0}G[0])(K)
\end{equation} 
(e.g. \cite[p. 912]{motiftate}). By a), the right hand term is the $n$-th homology group of the
complex
\[\dots\to T_n(K)\to\dots\to T_0(K)\to G(K)\to 0\]
with $G(K)$ in degree $0$. By \cite[p. 199, Th. 2]{cts}, the sequences 
\begin{gather*}
0\to S_1(K)\to T_0(K)\to T(K) \to T(K)/R \to 0\\
0\to S_{n+1}(K)\to T_n(K)\to S_n(K) \to S_n(K)/R \to 0
\end{gather*}
are all exact. Using Lemma \ref{l1} for $H_0$, the conclusion follows from an easy diagram
chase.
\end{proof}

\begin{rk} As a corollary to Theorem \ref{t3}, $S_n(K)/R$ only depends on $G$. This can be
seen without mentioning $\DM_-^\eff$: in view of the reasoning just above, it
suffices to construct a homotopy equivalence between two resolutions of the form \eqref{eq5},
which easily follows from the definition of coflasque modules.
\end{rk}

\begin{proof}[Proof of Corollary \ref{c2}] a) This follows via Theorem \ref{t3} and Lemma
\ref{l1} from \cite[p. 200, Cor. 2]{cts} and the Mordell-Weil-N\'eron theorem. b) We may choose
the $L_n$, hence the $S_n$ split by $E/k$. The conclusion now follows from Theorem \ref{t3}
and \cite[p. 200, Cor. 3]{cts}. The last claim is clear.
\end{proof}

\begin{rk} In characteristic $p>0$, all finitely generated perfect fields are finite. To give
some contents to Corollary \ref{c2} a) in this characteristic, one may pass to the perfect
[one should say radicial] closure $k$ of a finitely generated field $k_0$. If $G$ is a
semi-abelian
$k$-variety, it is defined over some finite extension $k_1$ of $k_0$. If $k_2/k_1$ is a finite
(purely inseparable) subextension of $k/k_1$, then the composition
\[G(k_2)\by{N_{k_2/k_1}} G(k_1)\to G(k_2)\]
equals multiplication by $[k_2:k_1]$. Hence Corollary \ref{c2} a) remains true at least after inverting $p$.
\end{rk}

\section{Stable birationality}\label{s4}

If  $X$ is a smooth variety over a field $k$, we write $\Alb(X)$ for its
generalised Albanese variety in the sense of Serre \cite{serrealb}: it is a semi-abelian
variety, and a rational point $x_0\in X$ determines a morphism $X\to \Alb(X)$ which is universal
for morphisms from $X$ to semi-abelian varieties sending $x_0$ to $0$.

We also write $\NS(X)$ for the group of cycles of codimension $1$ on $X$ modulo algebraic
equivalence. This group is finitely generated if $k$ is algebraically closed \cite[Th. 3]{picfini}.

\subsection{Well-known lemmas} I include proofs for lack of reference.

\begin{lemma}\label{l6} a) Let $G,G'$ be two semi-abelian $k$-varieties. Then any $k$-morphism
$f:G\to G'$ can be written uniquely $f = f(0)+f'$, where $f'$ is a homomorphism.\\
b) For any semi-abelian $k$-variety $G$, the canonical map $G\to \Alb(G)$ sending $0$ to $0$ is
an isomorphism.
\end{lemma}

\begin{proof} a) amounts to showing that if $f(0)=0$, then
$f$ is a homomorphism.  By an adjunction game, this is equivalent to b). Let us give two
proofs: one of a) and one of b).

\emph{Proof of a)}. We may assume
$k$ to be a universal domain. The staement is classical for abelian varieties
\cite[p. 41, Cor. 1]{mumford} and an easy computation for tori. In the general case, let $T,T'$
be the toric parts of $G$ and $G'$ and $A,A'$ be their abelian parts. Let
$g\in G(k)$. As any morphism from $T$
to $A'$ is constant, the $k$-morphism
\[\phi_g:T\ni t\mapsto f(g+t)-f(g)\in G'\]
(which sends $0$ to $0$) lands in $T'$, hence is a homomorphism. Therefore it only depends on the
image of $g$ in $A(k)$. This defines a morphism
$\phi:A\to
\uHom(T,T')$, which must be constant with value $\phi_0=f$. It follows that
\[(g,h)\mapsto f(g+h)-f(g)-f(h)\]
induces a morphism $A\times A\to T'$. Such a morphism is constant, of value $0$.

\emph{Proof of b)}.  This is true if $G$ is abelian, by rigidity and the equivalence between a) and b).
In general, any morphism from $G$ to an abelian variety is trivial on $T$. This shows that the
abelian part of $\Alb(G)$ is $A$. Let $T' = \Ker(\Alb(G)\to A)$. We also have the counit
morphism $\Alb(G)\to G$, and the composition $G\to \Alb(G)\to G$ is the identity. Thus $T$ is a
direct summand of $T'$. It suffices to show that $\dim T' = \dim T$. Going to the algebraic
closure, we may reduce to $T = \G_m$.

Then consider the line bundle completion $\bar G\to A$ of the $\G_m$-bundle $G\to A$. It is
sufficient to show that the kernel of
\[\Alb(G)\to \Alb(\bar G) = A\]
is $1$-dimensional. This follows for example from \cite[Cor. 10.5.1]{bvk}.
\end{proof}

\begin{lemma}\label{l9} Suppose $k$ algebraically closed, and let $G$ be a semi-abelian $k$-variety. Let
$A$ be the abelian quotient of $G$. Then the map
\begin{equation}\label{eq7}
\NS(A)\to \NS(G)
\end{equation}
is an isomorphism.
\end{lemma}

\begin{proof} Let $T=\Ker(G\to A)$ and $X(T)$ be its character group. Choosing a basis $(e_i)$ of
$X(T)$, we may complete the
$\G_m^n$-torsor $G$ into a product of line bundles $\bar G\to A$. The surjection 
\[\Pic(A)\iso \Pic(\bar G)\surj \Pic(G)\]
show the surjectivity of \eqref{eq7}. Its kernel is generated by the classes of the irreducible
components $D_i$ of the divisor with normal crossings $\bar G-G$. These components correspond to
the basis elements $e_i$. Since the corresponding $\G_m$-bundle is a group extension of $A$ by
$\G_m$, the class of the $0$ section of its line bundle completion lies in $\Pic^0(A)$, hence
goes to $0$ in $\NS(\bar G)$.
\end{proof}

\begin{lemma}\label{l5} Let $X$ be a smooth $k$-variety, and let
$U\subseteq X$ be a dense open subset. Then there is an exact sequence of semi-abelian varieties
\[0\to T\to \Alb(U)\to \Alb(X)\to 0\]
with $T$ a torus. If $\NS(\bar U)=0$ (this happens if $U$ is small enough), there is an exact
sequence of character groups
\[0\to X(T)\to \bigoplus_{x\in X^{(1)}-U^{(1)}} \Z\to \NS(\bar X)\to 0.\]
\end{lemma}

\begin{proof} This follows for example from \cite[Cor. 10.5.1]{bvk}.
\end{proof}

\begin{lemma}\label{l11} Let $f:G\tto G'$ be a rational map
between semi-abelian $k$-varieties, with
$G$ a torus.  Then there exists an extension $\tilde G$ of
$G$ by a permutation torus and a homomorphism
$\tilde f:\tilde G\to G'$ which extends $f$ up to translation in the following sense: there exists
a rational section $s:G\tto \tilde G$ of the projection
$\pi:\tilde G\to G$ and a rational point $g'\in G'(k)$ such that $f=\tilde f s+ g'$. If $f$ is
defined at $0_G$ and sends it to $0_{G'}$, then $g'=0$.
\end{lemma}

\begin{proof} Let $U$ be an open subset of $G$ where $f$ is defined. We define $\tilde
G=\Alb(U)$. Applying Lemmas \ref{l5} and \ref{l6} b) and using $\NS(\bar G)=0$,
we get an extension
\[0\to P\to \tilde G\to G\to 0\]
where $P$ is a permutation torus, as well as a morphism $\tilde f=\Alb(f):\tilde G\to G'$. 

Let us first assume $k$ infinite. Then $U(k)\neq \emptyset$ because $G$
is unirational. A rational point
$g\in U$ defines an Albanese map
$s:U\to
\tilde G$ sending $g$ to $0_{\tilde G}$. Since $P$ is a permutation torus, 
$g\in G(k)$ lifts to $\tilde g\in \tilde G(k)$ (Hilbert 90) and we may replace $s$ by a morphism
sending $g$ to
$\tilde g$. Then $s$ is a rational section of $\pi$. Moreover, $f=\tilde f s + g'$ with
$g'= f(g)-\tilde f(\tilde g)$. The last assertion follows.

If $k$ is finite, then $U$ has at least a zero-cycle $g$ of degree $1$, which is enough to define the Albanese map $s$. We then proceed as above (lift every closed point involved in $g$ to a closed point of $\tilde G$ with the same residue field).
\end{proof}

\begin{lemma} \label{l10} Let $G$ be a finite group, and let $A$ be a finitely generated
$G$-module. Then \\
a) There exists a short exact sequence of $G$-modules $0\to P\to F\to A\to 0$, with $F$
torsion-free and flasque, and $P$ permutation.\\
b) Let $B$ be another finitely generated $G$-module, and let $0\to P'\to E\to B\to 0$ be an
exact sequence with $P'$ an invertible module. Then any $G$-morphism $f:A\to B$ lifts to
$\tilde f:F\to E$.
\end{lemma}

\begin{proof} a) is the contents of \cite[Lemma 0.6, (0.6.2)]{cts2}. b) The obstruction to
lifting $f$ lies in $\Ext^1_G(F,P')=0$ \cite[p. 182, Lemme 9]{cts}.
\end{proof}

\subsection{Functoriality of $\nu_{\le 0} G$}\label{sp4} We now assume $k$ perfect.

\begin{lemma}\label{l4} Let 
\begin{equation}\label{eq4}
0\to P\to G\to H\to 0
\end{equation}
be an exact sequence of semi-abelian
varieties, with
$P$ an invertible torus. Then $\nu_{\le 0}G[0]\iso \nu_{\le 0} H[0]$.
\end{lemma}

\begin{proof} As $P$ is invertible, \eqref{eq4} is exact in $\NST$ hence defines an exact
triangle
\[P[0]\to G[0]\to H[0]\by{+1}\] 
in $\DM_-^\eff$. The conclusion then follows from Lemma \ref{l3}.
\end{proof}

\begin{prop}\label{p4} Let $G,G'$ be two semi-abelian $k$-varieties, with $G$ a torus. Then a
rational map
$f:G\tto G'$ induces a morphism $f_*:\nu_{\le 0} G[0]\to \nu_{\le 0} G'[0]$, hence a
homomorphism
$f_*:G(K)/R\to G'(K)/R$ for any extension $K/k$. If $K$ is infinite, $f_*$ agrees up to
translation with the morphism induced by $f$ via the isomorphism $U(K)/R\iso G(K)/R$ from
\cite[p. 196 Prop. 11]{cts}, where
$U$ is an open subset of definition of $f$.
\end{prop}

\begin{proof}  By Lemma
\ref{l11},
$f$ induces a homomorphism
$\tilde G\to G'$ where
$\tilde G$ is an extension of $G$ by a permutation torus. By Lemma
\ref{l4}, the induced morphism 
\[\nu_{\le 0} \tilde G[0]\to
\nu_{\le 0} G'[0]\] 
factors through a morphism $f_*:\nu_{\le 0} G[0]\to \nu_{\le 0} G'[0]$.

The claims about $R$-equivalence classes follow from Theorem \ref{t3} b) and Lemma \ref{l11}.
\end{proof}

\begin{rk} The proof shows that $f'_*=f_*$ if $f'$ differs from $f$ by a translation by an
element of $G(k)$ or $G'(k)$.
\end{rk}

\begin{cor}\label{c6} If $T$ and $T'$ are birationally equivalent $k$-tori, then $\nu_{\le 0}
T[0]\allowbreak\simeq
\nu_{\le 0} T'[0]$. In particular, the groups $T(k)/R$ and $T'(k)/R$ are isomorphic.
\end{cor}

\begin{proof} The proof of Proposition \ref{p4} shows that $f\mapsto f_*$ is functorial for
composable rational maps between tori. Let
$f:T\tto T'$ be a birational isomorphism, and let
$g:T'\tto T$ be the inverse birational isomorphism. Then we have $g_*f_*=1_{\nu_{\le 0}T[0]}$
and $f_*g_*=1_{\nu_{\le 0}T'[0]}$. The last claim follows from Theorem \ref{t3}.
\end{proof}

\begin{rk}\label{r1} It is proven in \cite{cts} that a birational isomorphism of tori $f:T\tto
T'$ induces a set-theoretic bijection $f_*:T(k)/R\iso T'(k)/R$  (p. 197, Cor. to Prop. 11) and
that the \emph{group} $T(k)/R$ is abstractly a birational invariant of $T$ (p. 200, Cor.
4). The proof above shows that  $f_*$ is an isomorphism of
groups if $f$ respects the origins of $T$ and $T'$. This solves the question raised in
\cite[mid. p. 397]{cts}. The proofs of Lemma
\ref{l11} and Proposition \ref{p4} may be seen as dual to the proof of \cite[p. 189, Prop.
5]{cts}, and are directly inspired from it.
\end{rk}

\subsection{Faithfulness and fullness}\label{proof of t4}

\begin{prop}\label{p6} Let $f:G\tto G'$ be a rational map between semi-abelian varieties, with
$G$ a torus. Assume that the map
$f_*:G(K)/R\to G'(K)/R$ from Proposition \ref{p4} is identically $0$ when $K$ runs through the
finitely generated extensions of $k$. Then there exists a permutation torus $P$ and a
factorisation of $f$ as
\[G\overset{\tilde f}{\tto} P\by{g}G'\]
where $\tilde f$ is a rational map and $g$ is a homomorphism. If $f$ is a morphism, we may choose $\tilde f$ as a homomorphism.\\
Conversely, if there is such a factorisation, then $f_*:\nu_{\le 0} G[0]\to \nu_{\le 0} G'[0]$
is the $0$ morphism.
\end{prop}

\begin{proof} By Lemma
\ref{l11}, we may reduce to the case where $f$ is a morphism. Let
$K=k(G)$. By hypothesis, the image of the generic point $\eta_G\in G(K)$ is $R$-equivalent to
$0$ on $G'(K)$. By a lemma of Gille
\cite[Lemme II.1.1 b)]{gille}, it is directly $R$-equivalent to $0$: in other words, there
exists a rational map $h:G\times \A^1\tto G'$, defined in the neighbourhood of $0$ and $1$,
such that $h_{|G\times\{0\}} = 0$ and $h_{|G\times\{1\}} = f$.

Let $U\subseteq G\times \A^1$ be an open set of definition of $h$. The $0$ and $1$-sections of
$G\times \A^1\to G$ induce sections
\[s_0,s_1:G\to \Alb(U)\]
of the projection $\pi:\Alb(U)\to \Alb(G\times
\A^1)=G$ such that $\Alb(h)\circ s_0=0$ and $\Alb(h)\circ s_1=f$. If $P=\Ker\pi$, then
$s_0-s_1$ induces a homomorphism $\tilde f:G\to P$ such that the composition
\[G\by{\tilde f}P\to \Alb(U)\by{\Alb(h)} G'\]
equals $f$. Finally, $P$ is a permutation torus by Lemma \ref{l5}.

The last claim follows from Lemma \ref{l3}.
\end{proof}

\begin{proof}[Proof of Theorem \ref{t4}] a)  Take
$K=k(G)$. The image of the generic point
$\eta_G$ by
$f_K$ lifts to a (non unique) rational map $f:G\tto G'$. Using Lemma \ref{l11}, we may extend
$f$ to a homomorphism
\[\tilde f:\tilde G\to G'\]
where $\tilde G$ is an extension of $G$ by a permutation torus $P$. Since $\tilde
G(K)/R\allowbreak\iso G(K)/R$, we reduce to $\tilde G = G$ and $\tilde f=f$.

Let $L/k$ be a fonction field, and let $g\in G(L)$. Then $g$ arises from a morphism $g:X\to G$
for a suitable smooth model $X$ of $L$. By assumption on $K\mapsto f_K$, the diagram
\[\begin{CD}
G(K)/R @>f_K>> G'(K)/R\\
@V{g^*}VV @V{g^*}VV\\
G(L)/R @>f_L>> G'(L)/R
\end{CD}\]
commutes. Applying this to $\eta_K\in G(K)$, we find that $f_L([g])=[g\circ f]$, which means
that $f_L$ is the map induced by $f$.

b) The hypothesis implies that $G'(E)/R=0$ for any algebraically closed extension $E/k$, which
in turn implies that $G'$ is also a torus. Applying a), we may, and do, convert $f$ into a
true homomorphism by replacing $G$ by a suitable extension by a permutation torus. Applying
Lemma \ref{l10} a) to the cocharacter group of $G$,  we get a resolution
$0\to P_1\to Q\to G\to 0$  with $Q$ coflasque and $P_1$
permutation. Hence we may (and do) further assume $G$ coflasque.

Let
$K=k(G')$ and choose some
$g\in G(K)$ mapping modulo $R$-equivalence to the generic point of $G'$. Then $g$ defines a
rational map
$g:G'\tto G$ such that $fg$ is $R$-equivalent to $1_{G'}$. It follows that the induced map
\begin{equation}\label{eq8}
1 - fg:G'/R\to G'/R
\end{equation}
is identically $0$. 

Reapplying Lemma \ref{l11}, we may find an extension $\tilde G'$ of $G'$ by a suitable
permutation torus which converts $g$ into a true homomorphism. Since $G$ is
coflasque, Lemma \ref{l10} b) shows that $f:G\to G'$ lifts to $\tilde f:G\to \tilde G'$. 
 Then \eqref{eq8} is still identically $0$ when replacing $(G',f)$ by $(\tilde G',\tilde f)$.

Summarising: we have replaced the initial $G$ and $G'$ by suitable extensions by permutation
tori, such that $f$ lifts to these extensions and there is a homomorphism $g:G'\to G$ such that
\eqref{eq8} vanishes identically. Hence $1-fg$ factors through a permutation torus $P$ thanks to
 Proposition \ref{p6}.
Write $u:G'\to P$ and $v:P\to G'$ for homomorphisms such that $1-fg= vu$. Let $G_1 = G\times P$
and consider the maps
\[f_1=(f,v):G_1\to G',\qquad g_1 = \begin{pmatrix} g\\ u\end{pmatrix}:G'\to G_1.\]

Then $f_1g_1 = 1$ and $G'$ is a direct summand of $G_1$ as requested.
\end{proof}

\section{Some open questions}\label{s5}

\begin{qn}\label{q1} Are lemma \ref{l11} and Proposition \ref{p4} still true when $G$ is not a
torus?
\end{qn}

This is far from clear in general, starting with the case where $G$
is an abelian variety and $G'$ a torus. Let me give a positive answer in the case of an
elliptic curve.

\begin{prop} \label{p5} The answer to Question \ref{q1} is yes if the abelian part $A$ of $G$ is
an elliptic curve.
\end{prop}

\begin{proof} Arguing as in the proof of Proposition \ref{p4}, we get for an open subset
$U\subseteq G$ of definition for $f$ an exact sequence
\[0\to \G_m\to P\to \Alb(U)\to G\to 0 \]
where $P$ is a permutation torus. Here we used that $\NS(\bar G)\simeq \Z$, which follows from
Lemma
\ref{l9}. 

The character group $X(P)$ has as a basis the geometric irreducible components of codimension
$1$ of $G-U$. Up to shrinking $U$, we may assume that $G-U$ contains the inverse image $D$ of
$0\in A$. As the divisor class of $0$ generates $\NS(\bar A)$, $D$ provides a Galois-equivariant
splitting of the map $\G_m\to P$. Thus its cokernel is still a permutation torus, and we
conclude as before.
\end{proof}

\begin{qn} \label{q2} Can one formulate a version of Theorem \ref{t4} and Corollary \ref{c4}
providing a description of the groups $\Hom_{\DM_-^\eff}(\nu_{\le 0} G[0], \nu_{\le 0} G'[0])$
and
$\Hom_\HI(G/R,G'/R)$ (at least when
$G$ and
$G'$ are tori)?
\end{qn}

The proof of Theorem \ref{t4} suggests the presence of a closed model structure on the category
of tori (or lattices), which might provide an answer to this question.

For the last question, let $G$ be a semi-abelian variety. Forgetting its group structure, it
has a motive $M(G)\in \DM_-^\eff$. Recall the canonical morphism 
\[M(G)\to G[0]\]
induced by the ``sum" maps
\begin{equation}\label{eq9}
c(X,G)\by{\sigma} G(X)
\end{equation}
for smooth varieties $X$ (\cite[(6), (7)]{spsz}, \cite[\S 1.3]{bvk}).

The morphism \eqref{eq9} has a canonical section 
\begin{equation}\label{eq10}
G(X)\by{\gamma} c(X,G) 
\end{equation}
given by the graph of a morphism: this section is functorial in $X$ but is not additive.

Consider now a smooth equivariant compactification $\bar G$ of $G$. It exists in all
characteristics. For tori, this is written up in \cite{cths}. The general case reduces to this
one by the following elegant argument I learned from M. Brion: if $G$ is an extension of an
abelian variety $A$ by a torus $T$, take a smooth projective equivariant compactification $Y$
of $T$. Then the bundle $G \times^T Y$  associated to the $T$-torsor $G
\to A$ also exists: this is the desired compactification.

Then we have a diagram of birational motives
\begin{equation}\label{eq11}
\begin{CD}
\nu_{\le 0} M(G)@>\sim>> \nu_{\le 0} M(\bar G)\\
@V{\nu_{\le 0}\sigma}VV\\
\nu_{\le 0} G[0].
\end{CD}
\end{equation}

By \cite{birat}, we have $H_0(\nu_{\le 0} M(\bar G))(X) = CH_0(\bar G_{k(X)})$ for any smooth
connected $X$. Hence the above diagram induces a homomorphism
\begin{equation}\label{eq9b}
CH_0(\bar G_{k(X)})\to G(k(X))/R
\end{equation}
which is natural in $X$ for the action of finite correspondences (compare Corollary \ref{c3}).
One can probably check that this is the homomorphism of \cite[(17) p. 78]{merk}, reformulating
\cite[Proposition 12 p. 198]{cts}. Similarly, the set-theoretic map
\begin{equation}\label{eq10b}
G(k(X))/R\to CH_0(\bar G_{k(X)})
\end{equation}
of \cite[p. 197]{cts} can presumably be recovered as a birational version of \eqref{eq10},
using perhaps the homotopy category of schemes of Morel and Voevodsky \cite{mv}.

In \cite{merk}, Merkurjev shows that \eqref{eq9b} is an isomorphism for $G$ a torus of
dimension at most $3$. This suggests:

\begin{qn} Is the map $\nu_{\le 0}\sigma$ of Diagram \eqref{eq11} an isomorphism when $G$ is a
torus of dimension $\le 3$?
\end{qn}

In \cite{merk2}, Merkurjev gives examples of tori $G$ for which \eqref{eq10b} is not a
homomorphism; hence its (additive) left inverse \eqref{eq9b} cannot be an isomorphism.
Merkurjev's examples are of the form $G= R^1_{K/k}\G_m\times R^1_{L/k}\G_m$, where $K$ and $L$
are distinct biquadratic extensions of $k$. This suggests:

\begin{qn} Can one study Merkurjev's examples from the above viewpoint? More generally, what is
the nature of the map $\nu_{\le 0}\sigma$ of Diagram \eqref{eq11}?
\end{qn}

We leave all these questions to the interested reader.


\begin{thebibliography}{II}
\bibitem{bvk} L. Barbieri-Viale, B. Kahn {\it On the derived category of $1$-motives}, {\tt
arXiv:1009.1900}.
\bibitem{cths} J.-L. Colliot-Th\'el\`ene, D. Harari, A. Skorobogatov {\it Compactification
\'equivariante d'un tore (d'apr\`es Brylinski et K\"unnemann)}, Expo. Math. {\bf 23} (2005),
161--170.
\bibitem{cts} J.-L. Colliot-Th\'el\`ene, J.-J. Sansuc {\it La $R$-\'equivalence sur les
tores}, Ann. Sci. \'Ec. Norm. Sup. {\bf 10} (1977), 175--230.
\bibitem{cts2} J.-L. Colliot-Th\'el\`ene, J.-J. Sansuc {\it Principal homogeneous spaces under
flasque tori; applications}, J. Alg. {\bf 106} (1987), 148--205.
\bibitem{deglisegenerique} F. D\'eglise {\it Motifs g\'en\'eriques}, Rend. Sem. Mat. Univ.
Padova {\bf 119} (2008), 173--244.
\bibitem{gille} P. Gille {\it La $R$-\'equivalence pour les groupes alg\'ebriques r\'eductifs
d\'efinis sur un corps global}, Publ. Math. IH\'ES {\bf 86} (1997), 199-235.
\bibitem{gille2} P. Gille {\it Sp\'ecialisation de la $R$-\'equivalence pour les groupes r\'eductifs}, Trans. Amer. Math. Soc. {\bf 356} (2004),  4465--4474.
\bibitem{motiftate} A. Huber and B. Kahn {\it The slice filtration and mixed Tate motives},
Compositio Math. {\bf 142} (2006), 907--936.
\bibitem{picfini} B. Kahn {\it Sur le groupe des classes d'un sch\'ema arithm\'etique} (avec un
appendice de Marc Hindry), Bull. Soc. Math. France {\bf 134} (2006), 395--415.
\bibitem{sk} B. Kahn, T. Yamazaki {\it Somekawa's $K$-groups and Voevodsky's Hom groups},  {\tt
arXiv:1108.2764}.
\bibitem{birat} B. Kahn, R. Sujatha {\it Birational motives, I (preliminary version)},
preprint, 2002, {\tt http://www.math.uiuc.edu/K-theory/0596/}.
\bibitem{merk} A. S. Merkurjev {\it $R$-equivalence on three-dimensional tori and zero-cycles}, 
Algebra Number Theory {\bf 2}  (2008), 69--89.
\bibitem{merk2} A. S. Merkurjev {\it Zero-cycles on algebraic tori}, {\it in}  The geometry of
algebraic cycles,  119--122, Clay Math. Proc., {\bf 9}, Amer. Math. Soc., Providence, RI, 2010.
\bibitem{mv} F. Morel, V. Voevodsky {\it $\A^1$-homotopy theory of schemes}, Publ. Math. IH\'ES {\bf 90} (1999), 45--143.
\bibitem{mumford} D. Mumford Abelian varieties (corrected reprint), TIFR -- Hindustan Book
Agency, 2008.
\bibitem{serrealb}  J.-P. Serre {\it Morphismes universels et vari\'et\'es d'Albanese}, 
in {\it Expos\'es de s\'eminaires, 1950--1989}, Doc. math\'ematiques {\bf 1}, SMF, 2001, 141--160.
\bibitem{spsz} M. Spiess, T. Szamuely {\it On the Albanese map for smooth quasi-projective
varieties}, Math. Ann. {\bf 325} (2003), 1--17.
\bibitem{voepre} V. Voevodsky {\it Cohomological theory of presheaves with transfers}, {\it
in} E. Friedlander, A. Suslin, V. Voevodsky Cycles, transfers and motivic cohomology theories,
Ann. Math. Studies {\bf 143}, Princeton University Press, 2000, 88--137. 
\bibitem{voetri} V. Voevodsky {\it Triangulated categories of motives
over a field}, {\it in} E. Friedlander, A. Suslin, V. Voevodsky Cycles,
transfers and motivic cohomology theories, Ann. Math. Studies {\bf 143},
Princeton University Press, 2000, 188--238. 
\end{thebibliography}
\end{document}